\documentclass[12pt,twoside]{amsart}
\usepackage {amssymb,latexsym,amsthm,amsmath}

\makeatletter
\def\imod#1{\allowbreak\mkern10mu({\operator@font mod}\,\,#1)}
\makeatother

\newtheorem{theorem}{Theorem}[section]
\newtheorem{lemma}[theorem]{Lemma}

\newtheorem{proposition}[theorem]{Proposition}

\newtheorem{definition}[theorem]{Definition}
\newtheorem*{theorem*}{Theorem}
\theoremstyle{definition}
\newtheorem{remark}[theorem]{Remark}

\newtheorem{examples}[theorem]{Examples}

\numberwithin{equation}{section}

\newcommand{\ignore}[1]{}

\newcommand{\mynote}[1]{}
\begin{document}
\setcounter{section}{0}
% document information
\title{Finiteness of $z$-classes in reductive groups}
\author{Shripad M. Garge.}
\address{Department of Mathematics, Indian  Institute  of  Technology  Bombay, Powai, Mumbai.  400 076.
INDIA.}
\email{shripad@math.iitb.ac.in}
\email{smgarge@gmail.com}
\author{Anupam Singh.}
\address{IISER Pune, Dr. Homi Bhabha Road, Pashan, Pune 411008, India.}
\email{anupamk18@gmail.com}
\date{}

\begin{abstract}
Let $k$ be a perfect field such that for every $n$ there are only finitely many field extensions, up to isomorphism, of $k$ of degree $n$.
If $G$ is a reductive algebraic group defined over $k$, whose characteristic is very good for $G$, then we prove that $G(k)$ has only finitely many $z$-classes. 

For each perfect field $k$ which does not have the above finiteness property we show that there exist groups $G$ over $k$ such that $G(k)$ has infinitely many $z$-classes.
\end{abstract}

\maketitle

\section{Introduction}
Let $G$ be a group.
A $G$-set is a non-empty set $X$ admitting an action of the group $G$. 
Two $G$-sets $X$ and $Y$ are called $G$-isomorphic if there is a bijection $\phi: X \to Y$ which commutes with the $G$-actions on $X$ and $Y$.
If the $G$-sets $X$ and $Y$ are transitive, say $X = G x$ and $Y = G y$ for some $x \in X$ and $y \in Y$, then $X$ and $Y$ are $G$-isomorphic if and only if the stabilizers $G_x$ and $G_y$ of $x$ and $y$, respectively, in $G$ are conjugate as subgroups of $G$. 

One of the most important group actions is the conjugation action of a group $G$ on itself. 
It partitions the group $G$ into its conjugacy classes which are orbits under this action. 
However, from the point of view of $G$-action, it is more natural to partition the group $G$ into sets consisting of conjugacy classes which are $G$-isomorphic to each other.
This motivates the following definition:

\begin{definition}
Let $G$ be a group and let $x, y \in G$. 
We say that $x$ and $y$ are {\em $z$-equivalent} if the centralisers of $x$ and $y$ are conjugate in $G$. 
\end{definition}

It is clear that $z$-equivalence is an equivalence relation. 
The equivalence classes under $z$-equivalence are called {\em $z$-classes}.
Each $z$-class is union of certain conjugacy classes of $G$, more precisely, each $z$-class is union of the conjugacy classes in $G$ which are $G$-isomorphic to a given conjugacy class.
A group $G$ is abelian if and only if it has a single $z$-class consisting of all its elements.
In general, the center of a group $G$ is the $z$-class of the identity element.

The notion of a $z$-class was introduced to the authors by Ravi Kulkarni. 
He and his students have computed $z$-classes for some classical groups (for instance, see \cite{Go} and \cite{GK}).
These computations for the unitary groups are done by Bhunia and Singh (see \cite{BSi}) and for the compact real group of type $G_2$ by Singh \cite{Si}.

The number of $z$-classes of a finite group is finite. 
It is, however, interesting to note that there exist infinite groups which have finitely many $z$-classes.
For instance, the infinite dihedral group, $D_{\infty} = \langle r, s: s^2 = 1, srs=r^{-1}\rangle$, has three $z$-classes consisting, respectively, of reflections, the non-central rotations and the central rotations.
Indeed, the other infinite dihedral group, the uncountable one, the group $O_2 (\mathbb{R})$ of $2 \times 2$ orthogonal matrices with real entries, also has the same description of its $z$-classes.

The $z$-classes in $GL_n(\mathbb{C})$ are also known as orbit types. 
It follows from the theory of Jordan canonical forms that in $GL_n(\mathbb{C})$ there are infinitely many conjugacy classes but only finitely many $z$-classes. 
For instance, there are only three $z$-classes in $GL_2(\mathbb{C})$.
We prove below that, in general, the number of $z$-classes is finite for a reductive group over an algebraically closed field (see Theorem~\ref{Fz-kbar}). 

On the other hand, the number of $z$-classes in the group $GL_2(\mathbb{Q})$ is infinite. 
Indeed, non-isomorphic quadratic extensions of $\mathbb{Q}$ give rise to non-isomorphic maximal tori in $GL_2(\mathbb{Q})$ which further give rise to non-conjugate centralisers of regular semisimple elements in $GL_2(\mathbb{Q})$.
We give more details on this theme in the fifth section of this paper.

Thus, the arithmetic of the base field, $k$, seems to govern the (in)finiteness of the number of $z$-classes in a reductive group defined over $k$. 
The field of complex numbers has no non-trivial field extension, $\mathbb{R}$ has only one non-trivial field extension and the groups $GL_2(\mathbb{C})$ and $O_2(\mathbb{R})$ has finitely many $z$-classes whereas $\mathbb{Q}$ has infinitely many field extensions and the group $GL_2(\mathbb{Q})$ does have infinitely many $z$-classes. 
This phenomenon has been observed in the papers mentioned above as well. 
Note that each of those papers concentrates on a group of a fixed type.
Whenever the base field $k$ has only finitely many extensions of any given degree, up to isomorphism, then it is proved in the above papers that the corresponding groups have finitely many $z$-classes. 

Our aim in this paper is to prove this result in general. 
We consider a perfect field $k$ which has only finitely many extensions, up to isomorphism, of any given degree. 
Such fields are called {\em fields of type $(F)$} by Borel and Serre (\cite{BSe}). 
We prove that if $G$ is a reductive linear algebraic group over a field $k$ of type $(F)$ then $G(k)$ has finitely many $z$-classes. 
We give a general proof for all the cases, treating all the groups at one go, including the exceptional groups whose $z$-classes have not been studied from this point of view yet.
The proofs given in the works quoted above are computational (therefore they do not apply easily to the exceptional groups, especially of type $E_6, E_7$ and $E_8$) whereas our proof, we believe, is more conceptual.

In the next section, we discuss the $z$-classes in $G(\bar{k})$ where $G$ is a reductive group defined over an algebraically closed field $\bar{k}$.
We reproduce the proof, due mainly to Steinberg (\cite[Corollary 1, page 107]{St}), that such a $G(\bar{k})$ has finitely many $z$-classes.
The third section discusses some basic notions of Galois cohomology, our main tool, and introduce the notion of a field of type $(F)$.
The fourth section gives the proof of our main result.
The final section completes the paper with some observations regarding the non-finiteness (see Proposition~\ref{exampleA} and \ref{exampleB}) of $z$-classes in a reductive group defined over a field $k$ which is not of type $(F)$.

The main idea of our proof is as follows:
Let $X$ be the set of $z$-classes in $G(k)$ and let $Y$ be the set of $z$-classes in $G(\overline{k})$. 
We want to have a function $\theta: X \to Y$ which sends the $z$-class of an element in $G(k)$ to the $z$-class of the same element in $G(\overline{k})$. 
Assume, for the moment, that we have such a function. 
By the theorem of Steinberg and others, $Y$ is known to be a finite set and Galois cohomology allows us to prove that the fibre of each element of $Y$ is a finite subset of $X$. 
This would prove that $X$ is finite, provided that we have the function $\theta: X \to Y$.

Interestingly, that is not so.
This means that there exists an example of a group $G$ defined over a field $k$ and elements $g, h \in G(k)$ such that $g, h$ are $z$-equivalent in $G(k)$ but not in $G(\overline{k})$. 

\begin{examples}
Let $G$ be the standard Borel subgroup in $GL_2$, the group of $2 \times 2$ upper triangular matrices, over the field $\mathbb{F}_2$. 
The order of $G(\mathbb{F}_2)$ is two, hence it is abelian. 
Therefore, the identity element in $G(\mathbb{F}_2)$ is $z$-equivalent to the non-trivial unipotent element, say $g$, in $G(\mathbb{F}_2)$. 
But, as is well-known, the centralizer of the element $g$ in $G(\overline{\mathbb{F}}_2)$ is of dimension two, generated by the unipotent elements and the central elements of $G(\overline{\mathbb{F}}_2)$. 
Hence $g$ is not $z$-equivalent to the trivial element in $G(\overline{\mathbb{F}}_2)$. 

If we let $G$ to be the standard Borel subgroup in $SL_2$ over the field $\mathbb{F}_3$ then we again get an example of $g, h$ which are $z$-equivalent in $G(\mathbb{F}_3)$ but not in $G(\overline{\mathbb{F}}_3)$. 
\end{examples}

\begin{remark}{\rm 
We believe that there are very few examples of the above kind.
In particular, we believe that if $G$ is a reductive group defined over a field $k$ then $G(k)$ will never give such an example. 
More precisely, if $G$ is a reductive group defined over a field $k$ and $g, h \in G(k)$ are $z$-equivalent in $G(k)$ then we believe that $g, h$ are $z$-equivalent in $G(L)$ for every field extension $L$ of $k$. 
But we are not going to prove it here. 
We intend to take it up later in a subsequent paper. }
\end{remark}

Now, to salvage the situation we make two cases in the proof of the main result. 
In the first case, we consider the field $k$ to be a finite field, where the finiteness of the number of $z$-classes of $G(k)$ for any algebraic group $G$ is clear. 
In the second case, we consider the field $k$ to be an infinite field, but still of type $(F)$, and then appeal to the theorem of Rosenlicht \cite{Ro} which says that $G(k)$ is dense in the algebraic group $G$ whenever $k$ is an infinite perfect field. 

To end the introduction, we also explain why we concentrate only on reductive groups. 
Sushil Bhunia recently proved in \cite{B} that if we take $G$ to be the group of upper triangular $n \times n$ matrices for $n \geq 6$ then the finiteness of the number of $z$-classes in $G(k)$ forces $k$ to be finite.
Thus, the finiteness of the number of $z$-classes is not true for non-reductive groups but it is true for reductive groups as we prove further in this paper. 

\section{$z$-classes in a reductive group over an algebraically closed field.}

In this section, we prove that $G(\bar{k})$, for a reductive linear algebraic group $G$ defined over an algebraically closed field $\bar{k}$, has only finitely many $z$-classes. 
We follow the proof of Steinberg from \cite[Cor. 1, page 107]{St} where only the good characteristics case was considered.
The proof we present below works for all characteristics.

\begin{theorem}\label{Fz-kbar}
Let $G$ be a reductive linear algebraic group defined over an algebraically closed field $\bar{k}$. 
The number of $z$-classes of $G(\bar{k})$ is finite.
\end{theorem}

\begin{proof}
For an element $g \in G(\bar{k})$, we denote by $Z(g)$ the centraliser of $g$ in $G(\bar{k})$. 

We first prove that there are only finitely many $z$-classes of semisimple elements in $G(\bar{k})$. 
Let $g \in G(\bar{k})$ be a semisimple element.
Choose a maximal torus $T$ of $G$ such that $g \in T(\bar{k})$.
Then $Z(g)$ can be described in terms of the roots of the torus $T$ which act trivially on the element $g$ and certain elements of the Weyl group (\cite[page 98]{St}). 
Since this is a finite data, the set of centralisers in $G$, $Z(t)$, of elements $t \in T$ is a finite set. 
Further, since any two maximal tori of $G$ are conjugate over $\bar{k}$ the set of centralisers in $G(\bar{k})$, $Z(g)$, of semisimple elements $g \in G(\bar{k})$ is a finite set up to conjugacy in $G(\bar{k})$.
Thus, there are only finitely many $z$-classes of semisimple elements in $G(\bar{k})$.

Now, a general element $g \in G(\bar{k})$ admits a Jordan decomposition $g = g_sg_u$ where $g_s$ and $g_u$ are unique elements in $G(\bar{k})$ satisfying $g = g_sg_u = g_ug_s$ and such that $g_s$ is semisimple and $g_u$ is unipotent (\cite[page 29]{St}). 
Then $Z(g) = Z(g_s) \cap Z(g_u)$ which is the same as the centraliser of the unipotent element $g_u$ in the group $Z(g_s)$, which is itself a reductive group (\cite[page 98]{St}). 
It is a theorem due to Richardson (\cite{Ri}) and Lusztig (\cite{Lu}) that a reductive group defined over an algebraically closed field has only finitely many conjugacy classes of unipotent elements.
Thus, in each $Z(g_s)$ there are only finitely many conjugacy classes of unipotent elements, and hence only finitely many $z$-classes of unipotent elements, while there are only finitely many $Z(g_s)$ up to conjugacy in $G(\bar{k})$. 

Hence there are only finitely many $z$-classes of elements in $G(\bar{k})$. 
\end{proof}

\begin{remark}
In what follows in the remaining sections of this paper, we will be considering centralizers of various elements in $G(k)$ and then consider the normalizers of these centralizers in $G(\overline{k})$. 
For computing the first cohomology groups of these normalizers, we will need them (as also the centralizers of elements themselves) to be smooth as group schemes over $\overline{k}$. 
We are therefore forced to assume that the characteristic, if positive, of the base field $k$ is ``very good'' for the group $G$. 
Whenever $G$ is a simple group, the characteristic $p > 0$ is {\em good} if $p \ne 2$ if $G$ is not of type $A$, $p \ne 3$ if $G$ is exceptional and $p \ne 5$ if $G$ is of type $E_8$. 
For a semisimple group $G$, the prime $p$ is {\em good} if it is good for all of its simple factors and $p$ is {\em very good} if it is good for $G$ and it does not divide $n+1$ for any factor of $G$ of type $A_n$.

By a theorem of Cartier, the characteristic zero is always good, very good and even pretty good for any reductive group $G$. 
Further, every prime $p$ is very good for a torus. 
We do not go into the definition of a pretty good prime but refer the reader to the papers \cite{MT09, HS, MT16} for the technical details regarding these restrictions. 

This restriction on the characteristic of the field $k$ is assumed throughout the remaining part of the paper. 
\end{remark}

\section{Galois cohomology}

In this section, we discuss basic Galois cohomology. 
In particular, we prove that the forms of $z$-classes in a group $G$ defined over a perfect field $k$ correspond to a subset of the first Galois cohomology set, $H^1(\overline{k}/k, H(\bar{k}))$, of a suitable linear algebraic group $H$.
We give a short exposition of Galois cohomology below and refer the reader to the book by Serre (\cite[Chapter III Section 1]{Se}) for more details.

If a group $G$ acts by group automorphisms on a group $A$ then we say that $A$ is a $G$-module.
One then forms the first Galois cohomology set, $H^1(G, A)$, in the following way.
The functions $\phi:G \to A$ satisfying $\phi(gh) = \phi(g)g(\phi(h))$ for all $g, h \in G$ are called $1$-cocycles on $G$ with values in $A$. 
Two $1$-cocycles $\phi, \psi$ on $G$ with values in $A$ are called {\em cohomologous}, denoted by $\phi \sim \psi$, if there exists $a \in A$ such that $\phi(g) = a^{-1}\psi(g)g(a)$ for all $g \in G$.
It follows that the $\sim$ above is an equivalence relation. 

The set of all $1$-cocycles on $G$ with values in $A$, modulo the equivalence relation $\sim$, is called the {\em first Galois cohomology set} of $G$ with values in $A$ and is denoted by $H^1(G, A)$. 
The equivalence class of the trivial $1$-cocycle, sending every element of $G$ to the identity in $A$, is called a distinguished point in $H^1(G, A)$. 
Any $G$-module homomorphism $f: A \to B$ induces a natural map $\tilde{f}:H^1(G, A) \to H^1(G, B)$ sending the distinguished element to the distinguished element.

If $A$ is a linear algebraic group defined over a perfect field $k$ then, for any Galois extension $L$ of $k$, the Galois group $Gal(L/k)$ acts on the group $A(L)$, of $L$-rational points of $A$, and one forms the set $H^1(Gal(L/k), A(L))$. 
This set is normally denoted by $H^1(L/k, A(L))$. 

\begin{lemma}
Let $G$ be a linear algebraic group defined over an infinite perfect field $k$. 
If two elements $g, h \in G(k)$ are $z$-equivalent over $k$ then they are also $z$-equivalent over $\overline{k}$. 

In particular, there exists a function $\theta$ that sends the $z$-class of $g \in G(k)$ to the $z$-class of $g$ in $G(\overline{k})$. 
\end{lemma}

\begin{proof}
Since $g, h$ are $z$-equivalent over $k$, the subgroups $Z_G(g)(k)$ and $Z_G(h)(k)$ are conjugate by an element, say $a \in G(k)$. 
By Rosenlicht's theorem (\cite{Ro}), the groups of $k$-rational points are Zariski-dense in the groups of $\overline{k}$-rational points for any linear algebraic group. 
It, therefore, follows that the subgroups $Z_G(g)(\overline{k})$ and $Z_G(h)(\overline{k})$ of $G(\overline{k})$ are conjugate by $a$. 
Thus, $g$ and $h$ are $z$-equivalent over $\overline{k}$. 

The extistence of the function $\theta$ now follows.
\end{proof}

Now, let $G$ be again a reductive group defined over an infinite perfect field $k$. 
For a subgroup $H$ of $G$, we denote by $[H]_k$ the $k$-conjugacy class of $H$ in $G$ and by $[H]_{\overline{k}}$ the $\overline{k}$-conjugacy class of $H$ in $G$. 
Let
$${\mathcal A}_g =\left\{[H]_k: H \leq G, H {\rm~is~defined~over~}k, H = a Z(g)a^{-1} {\rm~for~some~}a \in G(\overline{k})\right\}$$
the set of $k$-conjugacy classes of $k$-subgroups of $G$ which become conjugate to $Z(g)$ over $\overline{k}$.
Then we have the following result. 

\begin{lemma}\label{Fi-zc}
Let $G$ be a reductive group defined over a perfect infinite field $k$ and let $g \in G(k)$. 
Then ${\mathcal A}_g$ is in one-one correspondence with a subset of $H^1(\overline{k}/k, N_g(\overline{k}))$ where $N_g$ is the normaliser of $Z(g)$ in $G$. 
\end{lemma}

\begin{proof}
Since $g \in G(k)$ is fixed, we shall denote the set ${\mathcal A}_g$ simply by $\mathcal A$ and $N_g$ by $N$.
We define a map $\Phi: \mathcal A \to H^1(\overline{k}/k, N(\overline{k}))$ and prove that it is one-one. 
To that end, for every $[H]_k \in \mathcal A$, we first define a $1$-cocycle from $Gal(\overline{k}/k)$ to $N(\overline{k})$. 
Later we will use the natural map from the set of $1$-cocycles to the $H^1$ to define the map $\Phi$. 

Let $[H]_k \in \mathcal A$.
In particular, $H$ is conjugate to $Z(g)$ by an element of $G(\overline{k})$.
Let $a \in G(\bar{k})$ be one such element, so that $H = a Z(g)a^{-1}$.
If $\sigma \in Gal(\overline{k}/k)$ then applying $\sigma$ to the above equality we get
$$\sigma(H) = \sigma(a) \sigma (Z(g))\sigma(a^{-1}) .$$
Since $H$ and $Z(g)$ are defined over $k$, we get that $\sigma(H) = H$ and $\sigma(Z(g)) = Z(g)$ and hence $a^{-1}\sigma(a)$ normalises $Z(g)$. 
It follows that the function sending $\sigma$ to $a^{-1}\sigma(a)$ is a $1$-cocycle taking values in $N(\overline{k})$. 
We denote this element by $\phi(H)$. 

If we choose a different element $a' \in G(\overline{k})$ conjugating $H$ to $Z(g)$ then we get a $1$-cocycle which is cohomologus to $\phi(H)$. 
One also checks easily that if we take any $k$-subgroup $H_1$ of $G$ which is conjugate to $H$ by an element of $G(k)$ then the element $\phi(H_1)$ is also cohomologus to $\phi(H)$. 
This gives the required map $\Phi$.

It remains to prove that the map $\Phi$ is injective. 
Let $[H_1]_k$ and $[H_2]_k$ be in $\mathcal A$ so that the corresponding $1$-cocyles $\phi([H_1]_k): \sigma \mapsto a^{-1}\sigma(a)$ and $\phi([H_2]_k) : \sigma \mapsto a'^{-1}\sigma(a')$ are cohomologus. 
Then we have $c \in H(\bar{k})$ such that 
$$\phi([H_1]_k)(\sigma)= c^{-1}\phi([H_2]_k)(\sigma) \sigma(c) .$$
Then, $a^{-1}\sigma(a) = c^{-1} a'^{-1}\sigma(a') \sigma(c)$ for all $\sigma$. 
This gives that the element $a'ca^{-1}$ is in $G(k)$, further it conjugates $H_1(k)$ to $H_2(k)$ and hence $H_1$ to $H_2$.
Thus, $[H_1]_k = [H_2]_k$. 
\end{proof}

For the purpose of proving the main result of this paper, the above result is enough, however, below we give an exact description of the subset of $H^1(\bar{k}/k, N_g(\bar{k}))$ which is in bijective correspondence with ${\mathcal A}_g$ in the theorem below.
This description will be useful in the last section where we discuss examples of $z$-classes.

\begin{theorem}
We continue with the notations of Lemma $3.2.$
Let $i:N_g(\bar{k}) \to G(\bar{k})$ denote the natural inclusion and $\tilde{i}:H^1(\bar{k}/k, N_g(\bar{k})) \to H^1(\bar{k}/k, G(\bar{k}))$ be the induced map. 
Then the subset of $H^1(\bar{k}/k, N_g(\bar{k}))$ which is in bijective correspondence with ${\mathcal A}_g$ is the set of elements of $H^1(\bar{k}/k, N_g(\bar{k}))$ which are mapped to the distinguished element of $H^1(\bar{k}/k, G(\bar{k}))$ under $\tilde{i}$.
\end{theorem}

We do not give the proof of this result but refer to the paper \cite[Lemma 2.1]{G} where a similar result is proved for maximal tori in a reductive group $G$. 
The same proof, verbatim, works in this case too.
The following result can also be proved by similar methods.

\begin{theorem}
If $G$ is a reductive linear algebraic group defined over a field $k$ and $g \in G(k)$ then the set of conjugacy classes in $G(k)$ which when base changed to $\bar{k}$ become equal to the conjugacy class of $g$ in $G(\bar{k})$ is in bijection with the subset of $H^1(\bar{k}/k, Z(g)(\bar{k}))$ consisting of elements which are mapped to the distinguished element under the natural map $\tilde{i}$.
\end{theorem}

\begin{definition}
A perfect field $k$ is called a {\em field of type $(F)$} if for any integer $n$ there exist only finitely many extensions of $k$ of degree $n$ (in a fixed algebraic closure $\overline{k}$ of $k$).
\end{definition}

This notion is introduced by Borel-Serre in \cite[III $\S$ 4]{BSe}. 
The examples of fields of type $(F)$ include the field $\mathbb{R}$ of real numbers, a finite field, a $p$-adic field and the field $\mathbb{C}((T))$ of formal power series in one variable over $\mathbb{C}$ whereas non-examples of such fields include number fields and $\overline{\mathbb{F}}_q((T))$, the latter because of the failure of the perfect-ness.

The fields of type $(F)$ are expected to have easier arithmetic. 
For instance, over a field of type $(F)$ one has only finitely many isomorphism classes of objects like division algebras, quadratic forms, and so on.
This follows from the following result proved by Borel and Serre.

\begin{theorem}[Th\'{e}or\`{e}me 6.2, Borel-Serre]\label{BoSe}
Let $k$ be a field of type $(F)$, and let $H$ be a linear algebraic group defined over $k$. 
The set $H^1(\overline{k}/k, H(\bar{k}))$ is finite. 
\end{theorem}

\section{Main theorems and their proofs}

\begin{theorem}\label{Fi-g}
Let $G$ be a reductive algebraic group defined over an infinite field $k$ of type $(F)$. 
Let $g \in G(k)$ and let $z_{g, \overline{k}}$ denote the $z$-class of $g$ in $G(\overline{k})$. 
The number of $z$-classes of elements in $G(k)$ which on the base change to $\bar{k}$ equal $z_{g, \overline{k}}$ is finite.
\end{theorem}

\begin{proof}
The $z$-class of an element in $G(k)$ is basically the conjugacy class of the centralizer subgroup of that element in $G(k)$. 
The $z$-classes of elements of $G(k)$ which, over $\overline{k}$, equal $z_{g, \overline{k}}$ is a subset of $\mathcal{A}_g$. 
By Lemma~\ref{Fi-zc}, the set ${\mathcal A}_g$ is in a one-one correspondence with a subset of $H^1(\overline{k}/k, N_g(\overline{k}))$.
Since $k$ is a field of type $(F)$, by Theorem~\ref{BoSe}, the set $H^1(\overline{k}/k, N_g(\bar{k}))$ is finite and hence ${\mathcal A}_g$ is finite.  
\end{proof}

\begin{remark}
With above notations, the set of the $z$-classes of elements of $G(k)$ which, over $\overline{k}$, equal $z_{g, \overline{k}}$ can be a proper subset of $\mathcal{A}_g$. 
We give an example towards the end of the next section explaining this. 
\end{remark}

\begin{theorem}\label{main-theorem}
Let $k$ be a field of type $(F)$ and let $G$ be a reductive algebraic group defined over $k$. 
The group $G(k)$ has only finitely many $z$-classes.
\end{theorem}

\begin{proof}
The theorem is clear if $k$ is a finite field. 

Now we assume that $k$ is an infinite field.
Let $X$ denote the set of $z$-classes in $G(k)$ and $Y$ denote the set of $z$-classes in $G(\overline{k})$.  
There is a natural map, by Lemma 3.1, $\theta: X \to Y$ which sends the $z$-class of $g$ in $G(k)$ to the $z$-class of $g$ in $G(\overline{k})$. 
By Theorem~\ref{Fz-kbar}, the set $Y$ is a finite set.
Further, by the previous theorem, Theorem~\ref{Fi-g}, each fiber of $\theta$ is a finite set. 
Hence $X$ itself is a finite set.
\end{proof}

\section{Examples}
This section deals with various examples of computations of $z$-classes. 
We first show that if a perfect field $k$ is not of type $(F)$ then the main theorem of this paper fails for $k$ in the sense that there exists a reductive group $G$ defined over $k$, of type $A$, in fact, such that $G(k)$ has infinitely many $z$-classes. 

\begin{lemma}
Let $k$ be a perfect field and $G=SL_n$ or $GL_n$. 
Fix an element $g\in G(k)$. 
Then ${\mathcal A}_g$ (the set of $k$-conjugacy classes of $k$-subgroups of $G$ which are $\overline{k}$-conjugate to $Z(g)$), is in one-one correspondence with $H^1(\overline k/k, N_g(\overline k))$ where $N_g=N(Z(g))$.

Further, the set of conjugacy classes in $G(k)$ which when base changed to $\bar{k}$ become equal to the conjugacy class of $g$ in $G(\bar{k})$ are in bijection with the set $H^1(\overline k/k, Z(g)(\overline k))$.
\end{lemma}

\begin{proof}
This follows from Theorem 3.3 as $H^1(\bar{k}/k, G(\bar{k}))=1$ when $G=GL_n$ or $SL_n$.
\end{proof}

We apply this lemma to certain regular semisimple and unipotent elements. 

\subsection{$z$-classes of semisimple elements}
Let us begin with semisimple elements. 
Fix a regular semisimple element $g\in GL_n(k)$ or in $SL_n(k)$. 
The centraliser $Z(g)$ over $\overline k$ is a maximal torus, say $T$.
The $z$-classes over $k$ which, after base change to $\overline k$, are equivalent to that of $g$ is given by $H^1(\overline k/k, N(T)(\bar{k}))$. 
This corresponds to the conjugacy classes of maximal tori over $k$, by \cite[Lemma 2.1]{G}.
Note that, although the statement in \cite{G} deals only with a split maximal torus, the argument works for any fixed maximal torus in $G$.

\begin{proposition}\label{exampleA}
Let $k$ be a perfect field not of type $(F)$, then there exists $n$ such that the group $GL_n(k)$ has infinitely many $z$-classes of semisimple elements.
\end{proposition}

\begin{proof}
When the field $k$ is not of type $(F)$, there exists some $n$ such that the field $k$ has infinitely many extensions of degree $n$. 
We claim that $H^1(\overline k/k, N(T)(\bar{k}))$ is infinite where $T$ is a maximal torus in $GL_n(\overline k)$. 
This can be explicitly seen as follows. 

Any extension $K/k$ of degree $n$ gives rise to a maximal torus $T$ of $GL_n(k)$ obtained by the process of Weil restriction of scalars, denoted by $R_{K/k}(\mathbb{G}_m)$. 
The group $T(k)$ is precisely the embedding of $K^{\times}$ in $GL_n(k)$ obtained by the left regular action of $K$ on itself, after chosing a basis of $K$ over $k$.
The maximal tori thus obtained are non-conjugate if the field extensions are not isomorphic. 
Thus, we obtain infinitely many $z$-classes of semisimple elements.
\end{proof}

The same is true for $SL_n(k)$ as well where the maximal tori can be thought of as $K^1$, the norm $1$ elements of $K$. 
In particular, if we take $k=\mathbb Q$, then the group $GL_n(\mathbb Q)$ and $SL_n(\mathbb Q)$ have infinitely many $z$-classes of semisimple elements as there are infinitely many non-conjugate maximal tori, representing $z$-classes of regular semisimple elements within them.

\subsection{$z$-classes of unipotent elements}
Let us consider the unipotent elements of the form 
$$u_{\beta}= \begin{pmatrix} 1 & \beta & 0 &0 &\cdots& 0 \\ &1 & 1 &0 & \cdots & \\ &&\ddots &\ddots &\ddots& \\ &&&1&1&0\\ &&& &1 &1 \\ & &&&&1 \end{pmatrix}$$ 
in $SL_n(k)$ where $\beta\in k^*$. We note that all $u_{\beta}$ are conjugate over $\overline k$ to $u_1$ which is a regular unipotent element.

\begin{proposition}\label{exampleB}
The set of $k$-forms of the conjugacy class of unipotents $u_{\beta}$ in $SL_n(k)$ is in one-one correspondence with $k^*/(k^*)^n$. 

For $n=2$, they form a single $z$-class, and when $n=3$, the $z$-classes are in one-one correspondence to $k^*/(k^*)^3$.
\end{proposition}

\begin{proof}
We begin with $n=2$. 
Let us fix a $\beta\neq 0$. 
The centraliser group is given by
$$Z(u_{\beta})=\left\{\begin{pmatrix}a& b\\ 0&a \end{pmatrix} : a^2=1\right\} \cong U\times \mu_2$$ 
where $U$ is the unipotent upper triangular matrix group and $\mu_2$ is the center of $SL_2(k)$. 
Further, $N(Z(u_{\beta}))=B$, is the Borel subgroup of upper-triangular matrices in $SL_2(k)$. 
By Lemma $5.1$, the set of $k$-forms of the conjugacy class of $u_{\beta}$ in $SL_2(k)$ is given by $H^1(\overline k/k, Z(u_{\beta})(\bar{k}))$. 
To compute this, we use the sequence 
$$1\rightarrow U \rightarrow Z(u_{\beta}) \rightarrow \mu_2 \rightarrow 1$$
and get $H^1(\overline k/k, Z(u_{\beta})(\bar{k}))=H^1(\overline k/k,\mu_2) \cong k^*/(k^*)^2$. 
However, the $z$-classes of $u_{\beta}$ in $SL_2(k)$ are given by $H^1(\overline k/k, B)=1$ as $B\cong \mathbb G_m\ltimes \mathbb G_a$. 
In~\cite{BSi} this is explicitly computed in Example 5.2.

The calculation for the general case when $n\geq 3$ is similar. 
Since $u_{\beta}$ is cyclic, the centraliser in $GL_n(k)$ is given by polynomials in $u_{\beta}$. 
Thus the centraliser of $u_{\beta}$ in $SL_n(k)$ is, 
$$Z(u_{\beta})= \left\{ \begin{pmatrix} 
\alpha_{1} & \alpha_{2}\beta & \alpha_{3}\beta &\cdots & \alpha_{n-1}\beta &\alpha_{n}\beta \\ 
&\alpha_{1} &\alpha_{2} &\alpha_{3} & \cdots & \alpha_{n-1}\\ 
&&\ddots &\ddots & \ddots& \\
&&& \ddots &\ddots& \alpha_{3}\\ 
&&& & \alpha_{1}&\alpha_{2} \\
&&&&&\alpha_{1}\end{pmatrix} : \alpha_{1}^n =1, \alpha_{i}\in k \right\}. 
$$
Hence, $Z(u_{\beta})=U\times \mu_n$ where $U$ is the unipotent radical of $Z(u_{\beta})$ and $\mu_n$ is the center of $SL_n(k)$. 
Once again, the set of $k$-forms of the conjugacy class of $u_{\beta}$ in $SL_n(k)$ is given by $H^1(\overline k/k, Z(u_{\beta})(\overline k))$. 
As before, using the sequence 
$$1\rightarrow U \rightarrow Z(u_{\beta}) \rightarrow \mu_n \rightarrow 1$$ 
we compute $H^1(\overline k/k, Z(u_{\beta})(\overline k)) = H^1(\overline k/k,\mu_n(\overline k)) \cong k^*/(k^*)^n$. 
However, the $z$-class of $u_{\beta}$ in $SL_n(k)$ is given by $H^1(\overline k/k, N(Z(u_{\beta}))(\bar{k}))$. 
For $n=3$ we get 
$$N(Z(u_{\beta})) = \left\{ \begin{pmatrix} a_{11} & a_{12} & a_{13} \\ & a_{22} & a_{23}\\ && a_{11}^{-1}a_{22}^{-1} \end{pmatrix} : a_{22}^3=1\right\}\cong U'\rtimes H'$$
where $U'$ is unipotent radical and $H'(\overline k)\cong \mathbb G_m\times \mu_3(\overline k)$. 
Thus, 
$$H^1(\overline k/k, N(Z(u_{\beta}))(\overline k))= H^1(\overline k/k, \mu_3(\overline k))\cong k^*/(k^*)^3 .$$ 
\end{proof}

Once again if we take $k=\mathbb Q$ then we get infinitely many conjugacy classes as well as $z$-classes corresponding to unipotent elements in $SL_3(\mathbb Q)$. 
More generally, when $n\geq 3$, $N(Z(u_{\beta}))=\tilde U \rtimes D$ is a subgroup of upper-triangular matrices, where $\tilde U$ is unipotent radical and $D$ consist of the diagonals normalising $Z(u_{\beta})$. 
A simple computation tells us that 
$$D=\{diag( d^{n-1} a, \ldots ,d^2 a, d a, a) \mid d^{\frac{n(n-1)}{2}}a^n=1\}\cong \mathbb G_m\times \mu_{\frac{n(n-1)}{2}} .$$ 
Thus, $SL_n(\mathbb Q)$ has infinitely many $z$-classes of unipotent elements for $n\geq 3$.

\subsection{$z$-classes in non-reductive groups}
Here, we show by means of an example that if $k$ is an algebraically closed field and $G$, defined over $k$, is not a reductive group then $G(k)$ can have infinitely many $z$-classes. 

Indeed, let $G$ be the  group of unipotent upper triangular $3\times 3$ matrices. 
For $t \in k^{\times}$, the centralizer $Z(h)$ in $G$ of the element $h=h(t)= \begin{pmatrix} 1& t & 0\\ 0& 1& 1\\ 0& 0& 1\end{pmatrix}$ consists of all matrices
$\begin{pmatrix} 1& ta& b \\ 0 &1&  a\\ 0 &0 & 1 \end{pmatrix}$ for $a,b \in k$. 
The centralizers $Z(h) = Z(h(t))$ are pairwise non-conjugate in $G(k)$, as $t$ varies through $k^{\times}$. 
So $G(k)$ has infinitely many distinct z-classes - e.g. those represented by the $h(t)$.

Further, as mentioned in the introduction, we also have the result of S. Bhunia (\cite{B}) that the group $G(k)$, where $G$ is the group of upper triangular $n \times n$ matrices over an infinite field $k$, has infinitely many $z$-classes whenever $n \geq 6$. 

\subsection{A curious example}
We close the paper with the following curious example which says that a $k$-form of a conjugacy class of a centralizer subgroup need not itself be a centralizer subgroup. 

Let $k$ be a perfect field and $G$ be a reductive group defined over $k$. 
For a $k$-subgroup $H$ of $G$, the set of $k$-forms of conjugacy classes of $H$ in $G$ is in bijection with a certain subset of $H^1(\bar{k}/k, N(H)(\bar{k})$, in accordance with the Theorem $3.2$.
If the field $k$ is of type $(F)$ then, by the theorem of Borel and Serre, one gets that the conjugacy class of $H$ in $G$ has only finitely many forms over $k$. 
Thus, in particular, the conjugacy classes of subgroups which arise as centralisers of some element in $G$ have only finitely many $k$-forms.

However, it is interesting to note that if $H=Z(g)$ for some $g \in G(k)$ then every $k$-form of the conjugacy class of $H$ may not be a conjugacy class of $Z(h)$ for any $h \in G(k)$.

Take $G = SL_n$ defined over the field $k$ of order $2$ and let $g \in G(k)$ be a regular semisimple element, for instance, one can take $g$ to be a generator of the torus $K^{\times}$ coming from a degree $n$ extension $K$ of $k$. 
Then $Z(g)$ is a maximal torus in $G$ and its $k$-forms are precisely the maximal tori in $G$.
The split maximal torus, say $T$, in $G$ is also one of them, however, there is no regular semisimple element over $k$ for $T$ in $G$ and hence $T$ is not equal to any element centraliser in $G(k)$.
In fact, $T(k)$ has only the trivial element and hence it is not equal to $Z(h)$ for any $h \in G(k)$.

\section*{Acknowledgements} 
We warmly acknowledge fruitful discussions with Dipendra Prasad, Maneesh Thakur and Anuradha Garge.
We also thank the referees for a careful reading of the paper and for many valuable suggestions improving the paper to a great extent.
The second named author was supported by NBHM through a research project and SERB MATRICS grant during this work.

\end{document}